\documentclass[11pt, reqno, a4paper]{amsart}

\usepackage{amsfonts}
\usepackage[T1]{fontenc}
\usepackage{enumerate}
\usepackage{amsmath}
\usepackage[latin1]{inputenc}
\usepackage{amssymb}
\usepackage{verbatim}
\usepackage{amsthm}
\usepackage{hyperref}
\usepackage[arrow, matrix, curve]{xy}
\usepackage{tikz-cd}
\usepackage{tikz}
\usetikzlibrary{arrows,matrix}

\usepackage{geometry}
\DeclareMathOperator{\Frob}{Frob}
\DeclareMathOperator{\fpt}{fpt}
\DeclareMathOperator{\Soc}{Soc}

\DeclareMathOperator{\Ker}{Ker}
\DeclareMathOperator{\Proj}{Proj}

\DeclareMathOperator{\Ext}{Ext}
\DeclareMathOperator{\Hom}{Hom}
\DeclareMathOperator{\ord}{ord}
\DeclareMathOperator{\Hyp}{Hyp}

\address{Susanne Müller\\ Johannes Gutenberg-Universität Mainz \\ Fachbereich 08 \\ Staudingerweg 9 \\ 55099 Mainz \\ Germany}
\email{susanne.mueller@uni-mainz.de}

\begin{document} 

\title{The $F$-pure threshold of quasi-homogeneous polynomials}
\author{Susanne Müller}
\begin{abstract}
\small{Inspired by the work of Bhatt and Singh \cite{BS} we compute the $F$-pure threshold of quasi-homogeneous polynomials.  
We first consider the case of a curve given by a quasi-homogeneous polynomial $f$ in three variables $x,y,z$ of degree equal to the degree of $xyz$ and then we proceed with the general case of a Calabi-Yau hypersurface, i.e. a hypersurface given by a quasi-homogeneous polynomial $f$ in $n+1$ variables $x_0, \ldots, x_n$ of degree equal to the degree of $x_0 \cdots x_n$.}
\end{abstract}
\maketitle

\newtheoremstyle{Remark}
  {5pt}                   
  {5pt}                   
  {\normalfont}           
  {}                      
  {\bfseries}             
  {.}                     
  {2mm}                   
  {}                      

\newtheorem{definition}{Definition}[section]
\newtheorem{theorem}[definition]{Theorem}
\newtheorem{proposition}[definition]{Proposition}
\newtheorem{lemma}[definition]{Lemma}

\newtheorem*{theorem_o}{Theorem} 

\theoremstyle{Remark}
\newtheorem{remark}[definition]{Remark}
\newtheorem{example}[definition]{Example}
\newtheorem{notation}[definition]{Notation}

\section{Introduction}

To any polynomial $f\in K[x_0, \ldots, x_n]=R$, where $K$ is a field of characteristic $p>0$, one can attach an invariant called the $F$-pure threshold, first defined in \cite{TakWat}, \cite{MusTakWat}.  
This invariant is the characteristic $p$ analogue of the log canonical threshold in characteristic zero. The $F$-pure threshold, which is a rational number, is a quantitative measure of the severity of the singularity of $f$. Smaller values of the $F$-pure threshold correspond to a "worse" singularity.

In this article we focus mainly on the computation of the $F$-pure threshold of quasi-homogeneous polynomials. A polynomial $f$ is called quasi-homogeneous if there exists an $\mathbb{N}$-grading of $K[x_0, \ldots, x_n]$ such that $f$ is homogeneous with respect to this grading. 
The first of our two main results is the following:

\begin{theorem_o}[see Theorem \ref{MainTheorem}]
Let $C=\Proj(R/fR)$ be the curve given by a quasi-homogeneous polynomial $f \in K[x,y,z]$ of degree equal to the degree of $xyz$ with an isolated singularity. Then
$$\fpt(f)= \begin{cases} 1, &\mbox{if  $C$ is ordinary} \\
1-\frac{1}{p}, & \mbox{otherwise.} \end{cases}$$
\end{theorem_o}

Here, a curve $C$ is (by definition) ordinary if and only if the map on $H^1(C, \mathcal{O}_C)$ induced by Frobenius is bijective. This theorem is a generalization of the two-dimensional case of the main theorem of Bhatt and Singh (\cite{BS}), which says that the $F$-pure threshold of an elliptic curve $E$ given by a homogeneous polynomial $f \in K[x,y,z]$ of degree three is $1$ if $E$ is ordinary and $1-\frac{1}{p}$ otherwise. In contrast to the paper of Bhatt and Singh our proof does not rely on deformation-theoretic arguments and hence gives a more elementary approach to this result.

More generally, we consider the case of a Calabi-Yau hypersurface \linebreak $X=\Proj \left(R/fR \right)$ given by a quasi-homogeneous polynomial $f \in K[x_0, \ldots, x_n]=R$ and relate the $F$-pure threshold of $f$ to a numerical invariant of $X$, namely the
order of vanishing of the so-called Hasse invariant.
For this, we consider the family $\pi: \mathfrak{X} \rightarrow \Hyp_{w}$ of hypersurfaces of degree $w=\deg\left( x_0 \cdots x_n \right)$ in the weighted projective space $\mathbb{P}^n \left( \alpha_0, \ldots, \alpha_n \right)$.
Our chosen hypersurface $X=\Proj\left(R/fR\right)$ gives a point $\left[X\right]$ in $\Hyp_{w}$.\\
The Hasse invariant $H$ of a suitable family $\pi: \mathfrak{X} \rightarrow S$ of varieties in characteristic $p$ (for the precise conditions see section 5) is the element in
\begin{align*}
\Hom \left(R^N \pi_{\ast}^{(1)} \mathcal{O}_{\mathfrak{X}^{(1)}}, R^N \pi_{\ast} \mathcal{O}_{\mathfrak{X}} \right)
 &\cong \Hom \left(\left( R^N \pi_{\ast} \mathcal{O}_{\mathfrak{X}} \right)^{p}, R^N \pi_{\ast} \mathcal{O}_{\mathfrak{X}} \right) \\
&\cong \Hom\left(\mathcal{O}_S, \left( R^N \pi_{\ast} \mathcal{O}_{\mathfrak{X}} \right)^{1-p} \right)\\
&\cong H^0 \left(S, \left( R^N \pi_{\ast} \mathcal{O}_{\mathfrak{X}} \right)^{1-p} \right)
\end{align*}
induced by the relative Frobenius $\Frob_{\mathfrak{X} / S}$. Here the relative Frobenius is given by the following diagram
\begin{equation*}
\begin{tikzcd} 
 \mathfrak{X} \arrow[bend right]{rdd}{\pi} \arrow[dashed]{rd}{\Frob_{\mathfrak{X} / S}} \arrow[bend left]{rrd}{\Frob_{\mathfrak{X}}} & & \\
			 & \mathfrak{X}^{(1)} \arrow{r} \arrow{d}{\pi^{(1)}} & \mathfrak{X} \arrow{d}{\pi} \\
       & S \arrow{r}{\Frob_S} & S
\end{tikzcd}
\end{equation*} 
Now, fix $s \in S$ and an integer $t>0$ and let $t[s]$ be the order $t$ neighbourhood of $s$. 
Then, the order of vanishing of the Hasse invariant at the point $s \in S$ is given by
$\ord_s(H)=\max  \left\{t | i^{\ast}H=0 \text{ where } i:t[s] \hookrightarrow S \right\}$.
The second main result of this paper, generalizing \cite{BS} to the quasi-homogeneous case, is the following:

\begin{theorem_o}[see Theorem \ref{Hassethm}]
If $p \geq w(n-2)+1$, then $\fpt(f)=1-\frac{h}{p}$, where $h$ is the order of vanishing of the Hasse invariant at $\left[X\right] \in \Hyp_{w}$ on the deformation space $\mathfrak{X}$ of $X \subset \mathbb{P}^n \left(\alpha_0, \ldots, \alpha_n \right)$.
\end{theorem_o}

In the homogeneous case this result was proven by Bhatt and Singh in \cite{BS} by pointing out a connection between the order of vanishing of the Hasse invariant and the injectivity of the map $H^{n-1}\left(X, \mathcal{O}_{X}\right) \stackrel{a_t}{\longrightarrow} H^{n-1}\left(tX, \mathcal{O}_{tX}\right)$ induced by $\Frob_R$, where $tX$ is the order $t$ neighbourhood of $X$ in $\mathbb{P}^n \left(\alpha_0, \ldots, \alpha_n \right)$. 
We generalize this statement to the quasi-homogeneous case using local cohomology instead of sheaf cohomology, i.e. we consider the map $a_t$ as a map $H^{n}_{\mathfrak{m}}\left(R/f\right)_0 \stackrel{a_t}{\longrightarrow} H^{n}_{\mathfrak{m}}\left(R/f^t\right)_0$, which makes our approach rather explicit.

We should also mention the paper \cite{HBWZ} of Hernández, Nú{\~n}ez - Betancourt, Witt and Zhang, where the authors compute the possible values of the $F$-pure threshold of a quasi-homogeneous polynomial of arbitrary degree using base $p$ expansions. In particular, as a corollary they get the same list of possible $F$-pure thresholds as we obtain here in the case of a Calabi-Yau hypersurface. 

In sections 2 and 3 we set up the notation and extend some results of Bhatt and Singh (\cite{BS}) to the quasi-homogeneous setting, leading to recovering the results of \cite{HBWZ} in the Calabi-Yau case. In section 4 we prove Theorem \ref{MainTheorem} by elementary methods. The final section 5 is dedicated to the proof of Theorem \ref{Hassethm}.

\textbf{Acknowledgements.}
I thank Manuel Blickle, Duco van Straten and Axel Stäbler for useful discussions and a careful reading of earlier versions of this article.
Furthermore, I thank the referee for a careful reading of this paper and many valuable comments, which helped to improve the paper. 
The author was supported by SFB/Transregio 45 Bonn-Essen-Mainz financed by Deutsche Forschungsgemeinschaft.

\section{Quasi-homogeneous polynomials with an isolated singularity}

In this section we fix some notation and give some basic facts, which we will need later. For further information we refer the reader to \cite{K}, for example.

Throughout this article, $K$ will denote a field of characteristic $p>0$ and $$R:=K[x]:=K[x_0, \ldots, x_n]$$ will be the polynomial ring over $K$ in $n+1$ variables. By $$\mathfrak{m}:= (x_0, \ldots, x_n)$$ we will denote the maximal ideal of $R$ generated by the variables of $R$.

An $(n+1)$-tuple $\alpha= \left(\alpha_0, \ldots, \alpha_n\right) \in \mathbb{N}_{>0}^{n+1}$ defines a grading on $R=K[x_0, \ldots, x_n]$ by setting
$$\deg(x_i):=\alpha_i \text{ and } \deg \left( x^k \right):=\sum_{i=0}^n \alpha_i k_i,$$
where $k=(k_0, \ldots, k_n)$ is a multi-index and $x^k= x_0^{k_0} \cdots x_n^{k_n}$. It follows that $R= \bigoplus_d R_d$, where 
$$R_d:=\left\{ \sum_{k=(k_0, \ldots, k_n)} p_{k} x^{k} \Bigr| \sum_{i=0}^n \alpha_i k_i=d, p_k \in K\right\}$$
and the elements of $R_d$ are called quasi-homogeneous polynomials of degree $d$ and type $\alpha$.  
For an element $f \in R_d$ we have
$$f\left( \lambda^{\alpha_0}x_0, \ldots, \lambda^{\alpha_n}x_n \right)=\lambda^d f(x_0, \ldots, x_n) \text{ for all } \lambda \in K.$$
We set $$w:=\sum_{i=0}^n \alpha_i = \deg \left( x_0 \cdots x_n \right).$$

A sequence $f_1, \ldots, f_m \in R$ is called a regular sequence if the image of $f_i$ in \linebreak $R/(f_1, \ldots, f_{i-1})$ is a non-zero-divisor ($1 \leq i \leq m$) and if $(f_1, \ldots, f_m) \neq R$.
We say that an element $f \in R_d$ has an isolated singularity if 
$$\frac{\partial f}{\partial x_0}, \ldots, \frac{\partial f}{\partial x_n}$$
is a regular sequence.
Furthermore, the Jacobian ideal of $f \in R$ is 
$$J(f):= \left(\frac{\partial f}{\partial x_0}, \ldots, \frac{\partial f}{\partial x_n} \right)$$
and the Milnor number of $f$ is defined by
$$\mu(f):= \dim_K R/J(f).$$

It is well-known that $f$ has an isolated singularity if and only if $\mu(f) < \infty$. This can be shown with explicit bounds for $\mu(f)$ by using the Poincaré series.
The Poincaré series $H_M(T)$ of a finitely generated graded $R$-module $M$ is defined by
$$H_M(T):= \sum_{j} \dim_{K} \left(M_{j}\right) T^j,$$
where $M_{j}$ is the homogeneous part of $M$ of degree $j$.

First, we want to compute $H_R(T)$. For this, we use the fact that if $f \in R$ is a homogeneous element of degree $d$, which is a non-zero divisor of $M$, then 
\begin{equation} \label{eqPoincare}
H_{M/fM}(T)=(1-T^d)H_M(T).
\end{equation}

Hence, $H_{R/x_n}(T)=\left(1-T^{\alpha_n}\right) H_R(T)$
and inductively we get
$$1=H_{R/\left(x_0, \ldots, x_n\right)}(T)=\prod_{i=0}^n \left(1-T^{\alpha_i}\right) H_R(T).$$
Thus,
$$H_R(T)=\prod\limits_{i=0}^n\frac{1}{\left(1-T^{\alpha_i}\right)}.$$

Now, let $f_0, \ldots, f_m \in R$ be quasi-homogeneous polynomials of type $\alpha$ with \linebreak $\deg(f_i)=d_i$, $0 \leq i \leq m$, such that $f_0, \ldots, f_m$ is a regular sequence. Then 
$$H_{R/(f_0, \ldots, f_m)}(T)= \prod\limits_{j=0}^m \left(1-T^{d_j}\right) H_R(T)
= \frac{\prod\limits_{j=0}^m \left(1-T^{d_j}\right)}{\prod\limits_{i=0}^n \left(1-T^{\alpha_i}\right)}.$$
It is shown in \cite[p. 213]{K} that for $m=n$ one has
$$\dim_K \left(R/(f_0, \ldots, f_n) \right) = \lim_{T \rightarrow 1} \prod\limits_{i=0}^n\frac{ \left(1-T^{d_i}\right)}{ \left(1-T^{\alpha_i}\right)} = \prod_{i=0}^n \frac{d_i}{\alpha_i}.$$
In particular, $R/(f_0, \ldots, f_n)$ is a finite dimensional $K$-algebra.

If we have $f \in R_d$ with an isolated singularity, then $\frac{\partial f}{\partial x_0}, \ldots, \frac{\partial f}{\partial x_n}$ form a regular sequence and we know that $\deg\left(\frac{\partial f}{\partial x_i}\right)=d-\alpha_i$. Therefore the Milnor number is given by
$$\mu(f) = \dim_K \left(R/J(f) \right) = \prod_{i=0}^n \frac{d-\alpha_i}{\alpha_i}.$$

\section{The $F$-pure threshold of a quasi-homogeneous polynomial}

In \cite{BS}, Bhatt and Singh explain how to compute the $F$-pure threshold of a homogeneous polynomial and consider in particular the case of a Calabi-Yau hypersurface. In the following section we extend their results to compute the $F$-pure threshold of a quasi-homogeneous polynomial. To a large extent the proofs are analogous to the ones in \cite{BS}.

During this section let $q=p^e$ be a power of $p$. Remember that $R=K[x_0, \ldots, x_n]$ and $\mathfrak{m}= (x_0, \ldots, x_n)$. By $F:R \rightarrow R$, $r \mapsto r^p$ we denote the Frobenius or $p$-th power map on $R$.
We denote by $$\mathfrak{a}^{[q]} := \left(a^q \big| a\in \mathfrak{a}\right)$$ the Frobenius power of an ideal $\mathfrak{a} \subset R$. 
For $f \in \mathfrak{m}$ one defines in \cite{BS} $$\mu_f(q):=\min \left\{k \in \mathbb{N} \big| f^k \in \mathfrak{m}^{[q]}\right\}$$
and observes that $\mu_f(1)=1$ and that 
\begin{equation*}
1 \leq \mu_f(q) \leq q.
\end{equation*}
Furthermore,
\begin{equation} \label{mu_f(pq)}
\mu_f(pq) \leq p \mu_f(q),
\end{equation}
since $f^{\mu_f(q)} \in \mathfrak{m}^{[q]}$ implies that $f^{p\mu_f(q)} \in \mathfrak{m}^{[pq]}$. Hence, $\left\{\frac{\mu_f(p^e)}{p^e}\right\}_{e \geq 0}$ is a non-increasing sequence of positive rational numbers and one defines:

\begin{definition}
The $F$-pure threshold of $f$ is
$$\fpt(f):= \lim_{e \rightarrow \infty} \frac{\mu_f(p^e)}{p^e}.$$
\end{definition}

The definition of $\mu_f(q)$ yields $f^{\mu_f(q)-1} \notin \mathfrak{m}^{[q]}$ and therefore we get $f^{p\mu_f(q)-p} \notin \mathfrak{m}^{[pq]}$. Together with \eqref{mu_f(pq)} we deduce $p\mu_f(q)-p+1 \leq \mu_f(pq) \leq p \mu_f(q),$ which implies
\begin{equation} \label{Abrundung}
\mu_f(q) = \left\lceil \frac{\mu_f(pq)}{p}\right\rceil .
\end{equation}

Now, let $f \in R=K[x_0, \ldots, x_n]$ be a quasi-homogeneous polynomial of degree $d$ and type $\alpha$ and let $t \leq q$ be an integer. Then the Frobenius iterate $F^e:R/fR \rightarrow R/fR$ lifts to a map $R/fR \rightarrow R/f^qR$. We compose this map with the canonical surjection $R/f^qR \rightarrow R/f^tR$ and get a map $\widetilde{F^e_t}$.

\begin{lemma} \label{interpretation mu}
Let $f \in R$ be a quasi-homogeneous polynomial of degree $d$ and type $\alpha$ and let $t \leq q$ be an integer.
Then $\mu_f(q)>q-t$ if and only if $\widetilde{F^e_t} : H^n_{\mathfrak{m}}(R/fR) \rightarrow H^n_{\mathfrak{m}}(R/f^tR)$ is injective.
\end{lemma}

\begin{proof} 
We have a commutative diagram with exact rows
\begin{equation*}
\begin{tikzcd}
      0 \arrow{r} & R(-d) \arrow{r}{f} \arrow{d}{f^{q-t}F^e}  & R \arrow{r} \arrow{d}{F^e} & R/fR \arrow{r} \arrow{d}{\widetilde{F^e_t}}& 0\\
      0 \arrow{r} & R(-dt) \arrow{r}{f^t}                 & R \arrow{r}              & R/f^tR  \arrow{r} & 0,
\end{tikzcd}
\end{equation*}
which gives an induced diagram of local cohomology modules
\begin{equation*}
\begin{tikzcd}
      0 \arrow{r} & H^n_{\mathfrak{m}}(R/fR) \arrow{r} \arrow{d}{\widetilde{F^e_t}}  & H^{n+1}_{\mathfrak{m}}(R)(-d) \arrow{r}{f} \arrow{d}{f^{q-t}F^e} & H^{n+1}_{\mathfrak{m}}(R) \arrow{r} \arrow{d}{F^e}& 0\\
      0 \arrow{r} & H^n_{\mathfrak{m}}(R/f^tR) \arrow{r} & H^{n+1}_{\mathfrak{m}}(R)(-dt) \arrow{r}{f^t} & H^{n+1}_{\mathfrak{m}}(R)  \arrow{r} & 0.
\end{tikzcd}
\end{equation*}

We first show that the map $F^e$ is injective. 
For this, it suffices to show that $F^{e}$ acts injectively on the socle
$$\Soc\left(H^{n+1}_{\mathfrak{m}}(R)\right)=\left\langle \left[\frac{1}{x_0 \ldots x_n}\right]\right\rangle,$$
which follows from $F^e \left(\left[\frac{1}{x_0 \ldots x_n}\right]\right)= \left[ \frac{1}{x_0^q \cdot \ldots \cdot x_n^q} \right] \neq 0$. 
Now by the five lemma $\widetilde{F^e_t}$ is injective if and only if $f^{q-t}F^e$ is injective. Again by looking at the socle one shows that $f^{q-t}F^e$ is injective if and only if 
$$f^{q-t}F^e \left(\left[\frac{1}{x_0 \ldots x_n}\right]\right)=\left[\frac{f^{q-t}}{x_0^q \ldots x_n^q}\right] \neq 0,$$
which is equivalent to $f^{q-t} \notin \mathfrak{m}^{[q]}$. By the definition of $\mu_f(q)$ this is equivalent to $\mu_f(q)>q-t$.
\end{proof}

This cohomological description of $\mu_f(q)$ will be very useful in the following sections. Now we want to compute the $F$-pure threshold of a quasi-homogeneous polynomial. As a first step we give a lower and an upper bound for $\mu_f(p)$ (see Lemma \ref{Abschaetzung p kein Teiler} and Lemma \ref{obere Abschaetzung}) from which one obtains bounds for $\mu_f(p^e)$. We start with some lemmata which are similar to results shown in \cite{HBWZ}.

\begin{lemma} \label{max Grad Jacobi}
Let $f \in R=K[x_0, \ldots, x_n]$ be a quasi-homogeneous polynomial of degree $d$ and type $\alpha$ with an isolated singularity. Then
$$R_{\geq (n+1)d-2w+1} \subset J(f),$$
where $w=\sum\limits_{i=0}^n \alpha_i$.
\end{lemma}

\begin{proof}
Since $f$ has an isolated singularity, there exists a $k \in \mathbb{N}$ such that $\mathfrak{m}^k \subset J(f)$. This means that $\left(R/J(f)\right)_i=0$ for all $i$ greater than some $N \in \mathbb{N}$. Thus, the Poincaré series 
$$H_{R/J(f)}(T)= \frac{\prod\limits_{j=0}^n \left(1-T^{d-\alpha_j}\right)}{\prod\limits_{i=0}^n \left(1-T^{\alpha_i}\right)}$$ 
must be a polynomial, which means that $\prod\limits_{j=0}^n \left(1-T^{d-\alpha_j}\right)$ is divisible by $\prod\limits_{i=0}^n \left(1-T^{\alpha_i}\right)$. Therefore
$$\deg \left(H_{R/J(f)}\right)=\left((n+1)d-\sum_{j=0}^n \alpha_j\right)-\sum_{i=0}^n \alpha_i=(n+1)d-2w.$$
This means that $\left(R/J(f)\right)_i=0$ for all $i \geq (n+1)d-2w+1$, thus, $R_{\geq (n+1)d-2w+1} \subset J(f)$.
\end{proof}

\begin{lemma} \label{lemma colon ideal}
We have
$$\left(\mathfrak{m}^{[q]} :_{R} R_{\geq (n+1)d-2w+1} \right) \setminus \mathfrak{m}^{[q]} \subset R_{\geq (q+1)w-(n+1)d}.$$
\end{lemma}

\begin{proof}
Suppose the statement is false, then there exists a monomial $$\lambda:= x_0^{q-1-b_0} \ldots x_n^{q-1-b_n} \subset \left(\mathfrak{m}^{[q]} :_{R} R_{\geq (n+1)d-2w+1} \right) \setminus \mathfrak{m}^{[q]}$$ of degree $$\deg(\lambda)= (q-1)w- \sum_{i=0}^n b_i \alpha_i < (q+1)w-(n+1)d.$$
Equivalently, $$\sum_{i=0}^n b_i \alpha_i > (q-1)w-(q+1)w+(n+1)d=(n+1)d-2w,$$
thus the monomial $\eta:=x_0^{b_0} \ldots x_n^{b_n}$ is an element of $R_{\geq (n+1)d-2w+1}$. Therefore,
$x_0^{q-1} \ldots x_n^{q-1} =\lambda \cdot \eta \in \mathfrak{m}^{[q]}$, which is a contradiction.
\end{proof}

This yields a lower and an upper bound for $\mu_f(q)$.

\begin{lemma} \label{Abschaetzung p kein Teiler}
Let $f \in R$ be a quasi-homogeneous polynomial of degree $d$ and type $\alpha$ with an isolated singularity. If $p \nmid \; \mu_f(q)$, then
$$\mu_f(q) \geq \frac{w(q+1)-nd}{d}.$$
\end{lemma}

\begin{proof}
With $k:= \mu_f(q)$ we have $f^k \in \mathfrak{m}^{[q]}$. The partial derivatives $\frac{\partial}{\partial x_i}$ map $\mathfrak{m}^{[q]}$ to $\mathfrak{m}^{[q]}$ and therefore $$k f^{k-1} \frac{\partial f}{\partial x_i} \in \mathfrak{m}^{[q]}$$ for all i. Since $k$ is non-zero in $K$, it follows $$f^{k-1} J(f) \subset \mathfrak{m}^{[q]}.$$ 
By the definition of $k$ we know that $f^{k-1} \notin \mathfrak{m}^{[q]}$ so that $f^{k-1} \in \left(\mathfrak{m}^{[q]} :_R J(f) \right) \setminus \mathfrak{m}^{[q]}$. By Lemma \ref{max Grad Jacobi} and Lemma \ref{lemma colon ideal} it follows that $$\left(\mathfrak{m}^{[q]} :_R J(f) \right) \setminus \mathfrak{m}^{[q]} \subset \left(\mathfrak{m}^{[q]} :_R R_{\geq (n+1)d-2w+1} \right) \setminus \mathfrak{m}^{[q]} \subset R_{\geq (q+1)w-(n+1)d}.$$
This means that $f^{k-1} \in R_{\geq (q+1)w-(n+1)d}$. Hence, $$d(k-1)=\deg(f^{k-1}) \geq w(q+1)-(n+1)d$$ and therefore $k \geq \frac{w(q+1)-nd}{d}$.
\end{proof}

\begin{lemma} \label{obere Abschaetzung}
Let $f \in R$ be a quasi-homogeneous polynomial of degree $d$ and type $\alpha$.
Then $$\mu_f(q) \leq \left\lceil \frac{wq-w+1}{d}\right\rceil$$ for all $q=p^e$.
\end{lemma}

\begin{proof}
For $k \in \mathbb{N}$ one has $f^k \in \mathfrak{m}^{[q]}$ if $dk \geq w(q-1)+1$. Thus, 
$$\mu_f(q)=\min \left\{k \big|f^k \in \mathfrak{m}^{[q]}\right\} \leq \left\lceil \frac{wq-w+1}{d}\right\rceil.$$
\end{proof}

The following lemma explains how to compute $\mu_f(pq)$ if $\mu_f(q)$ is given, when certain conditions are fulfilled.

\begin{lemma} \label{Liften mu_f}
Let $f \in R$ be a quasi-homogeneous polynomial of degree $d$ and type $\alpha$ with an isolated singularity.
\begin{enumerate}
	\item[(1)] If $\frac{\mu_f(q)-1}{q-1}=\frac{w}{d}$ for some $q=p^s$, then $\frac{\mu_f(pq)-1}{pq-1}=\frac{w}{d}$.
	\item[(2)] Suppose $p \geq nd-d-w+1$. If $\frac{\mu_f(q)}{q}<\frac{w}{d}$ for some $q=p^s$, then $\mu_f(pq)=p\mu_f(q)$.
	In particular,
$$\frac{\mu_f(pq)}{pq} = \frac{\mu_f(q)}{q},$$ thus the sequence $\left\{\frac{\mu_f(p^e)}{p^e}\right\}_{p^e \geq q}$ is constant.
\end{enumerate}
\end{lemma}

\begin{proof}
To prove the first statement, assume that $\frac{\mu_f(q)-1}{q-1}=\frac{w}{d}$ for some $q=p^e$. Then $f^{\mu_f(q)-1}$ has degree 
$$d (\mu_f(q)-1)=d \cdot \frac{w}{d} \cdot (q-1)=w(q-1).$$ 
Since we also know that $f^{\mu_f(q)-1} \notin \mathfrak{m}^{[q]}$, it follows that $f^{\mu_f(q)-1}$ generates \linebreak $\Soc\left(R/\mathfrak{m}^{[q]}\right)$. Therefore, $\left(f^{\mu_f(q)-1}\right)^{\frac{pq-1}{q-1}}$ generates $\Soc\left(R/\mathfrak{m}^{[pq]}\right)$, so
$$\left(\mu_f(q)-1\right)\frac{pq-1}{q-1}=\mu_f(pq)-1.$$ Rearranging the terms one obtains $\frac{\mu_f(pq)-1}{pq-1}=\frac{\mu_f(q)-1}{q-1}=\frac{w}{d}$.

In order to prove the second statement, note that by equation \eqref{mu_f(pq)} we only need to show that $\mu_f(pq) < p\mu_f(q)$ cannot occur.
Therefore, suppose $\mu_f(pq) < p\mu_f(q)$. First we will show that $\mu_f(pq)$ cannot be a multiple of $p$. Suppose $\mu_f(pq)=pl$ for some $l$. By equation \eqref{Abrundung} it follows that $p \mu_f(q)=pl=\mu_f(pq)$, which is a contradiction. We can now use Lemma \ref{Abschaetzung p kein Teiler} and get 
$$\mu_f(pq) \geq \frac{w(pq+1)-nd}{d}.$$
Using $\mu_f(pq) < p\mu_f(q)$ and the second assumption, namely $d\mu_f(q)<wq$, we get $$w(pq+1)-nd \leq d\mu_f(pq) \leq dp\mu_f(q)-d \leq wpq-p-d.$$
This gives $p \leq nd-d-w$, which contradicts our assumption on $p$.
\end{proof}

To see how one can calculate the $F$-pure threshold of a quasi-homogeneous polynomial using Lemma \ref{Abschaetzung p kein Teiler}, \ref{obere Abschaetzung} and \ref{Liften mu_f}, let us consider an example.

\begin{example}
Let $f=xy^2+x^4$, then $f$ is quasi-homogeneous of degree $8$ and type $\alpha=(2,3)$. Let $p \neq 2$. We claim that
$$\fpt(f)= \begin{cases} \frac{5}{8}, &\mbox{if  $p \equiv 1,3 \mod 8$} \\
\frac{5}{8}-\frac{1}{8p}, & \mbox{if  $p \equiv 5 \mod 8$}\\
\frac{5}{8}-\frac{3}{8p}, &\mbox{if  $p \equiv 7 \mod 8$.} \end{cases}$$

By Lemma \ref{obere Abschaetzung} we get 
$$\mu_f(p) \leq  \left\lceil \frac{5p-4}{8}\right\rceil=\left\lceil p-\frac{3p+4}{8}\right\rceil \leq p-1.$$
With Lemma \ref{Abschaetzung p kein Teiler} it follows $$\mu_f(p) \geq \left\lceil \frac{5p-3}{8}\right\rceil.$$
Since $\left\lceil \frac{5p-4}{8}\right\rceil \leq \left\lceil \frac{5p-3}{8}\right\rceil$, we conclude that $\mu_f(p)=\left\lceil \frac{5p-3}{8}\right\rceil$.

First, let $p \equiv 1 \mod 8$, then $\mu_f(p)=\frac{5}{8}p+\frac{3}{8}$. With the first part of Lemma \ref{Liften mu_f} it follows that $\mu_f(q)=\frac{5}{8}q+\frac{3}{8}$. Hence, $\fpt(f)=\lim\limits_{e \rightarrow \infty} \frac{\mu_f(p^e)}{p^e}=\frac{5}{8}$.

Secondly, let $p \equiv 3 \mod 8$, then $\mu_f(p)=\frac{5}{8}p+\frac{1}{8}$. Since we cannot use Lemma \ref{Liften mu_f}, we compute $\mu_f(p^2)$ in the same way we computed $\mu_f(p)$. We get $\mu_f(p^2)=\frac{5}{8}p^2+\frac{3}{8}$.  
With the first part of Lemma \ref{Liften mu_f} it follows $\mu_f(q)=\frac{5}{8}q+\frac{3}{8}$ for all $q=p^e \geq p^2$. Hence, $\fpt(f)=\lim\limits_{e \rightarrow \infty} \frac{\mu_f(p^e)}{p^e}=\frac{5}{8}$.

Now let $p \equiv 5 \mod 8$, then $\mu_f(p)=\frac{5}{8}p-\frac{1}{8}$. With the second part of Lemma \ref{Liften mu_f} it follows that $\left\{\frac{\mu_f(q)}{q}\right\}_q$ is a constant sequence. Hence, $\fpt(f)=\frac{5}{8}-\frac{1}{8p}$.

The last case is $p \equiv 7 \mod 8$. Here $\mu_f(p)=\frac{5}{8}p-\frac{3}{8}$. With the second part of Lemma \ref{Liften mu_f} it follows that $\left\{\frac{\mu_f(q)}{q}\right\}_q$ is a constant sequence. Hence, $\fpt(f)=\frac{5}{8}-\frac{3}{8p}$.
\end{example}

As a consequence of the above lemmata we obtain the following theorem, which is very similar to the one given in \cite{BS} for homogeneous polynomials.

\begin{theorem} \label{theorem f-schwelle}
Let $f \in K[x_0, \ldots, x_n]$ be a quasi-homogeneous polynomial of degree $w=\sum\limits_{i=0}^n \alpha_i$ with an isolated singularity. Then
\begin{enumerate}
	\item[(1)] $\mu_f(p)=p-h$, where $0 \leq h \leq n-1$ is an integer.
	\item[(2)] $\mu_f(pq)=p \mu_f(q)$ for all $q$ with $q \geq n-1$.
	\item[(3)] If $p \geq n-1$, then $\fpt(f)=1-\frac{h}{p}$, where $0 \leq h \leq n-1$.
\end{enumerate}
\end{theorem}

\begin{proof}
First, suppose $\mu_f(p)=p$. Then by Lemma \ref{Liften mu_f} (1) we get $\mu_f(q)=q$ for all $q$, so the first two assertions follow.

Now, suppose $\mu_f(p)<p$. Then by Lemma \ref{Abschaetzung p kein Teiler} it follows that
$$\mu_f(p) \geq \frac{w(p+1)-nw}{w}=p+1-n,$$
which gives $\mu_f(p)=p-h$ with $0 < h \leq n-1$. To prove the second assertion, suppose $\mu_f(pq) < p \mu_f(q)$ (see equation \eqref{mu_f(pq)}). Then $p \nmid \mu_f(pq)$, since otherwise $\mu_f(pq)=p \mu_f(q)$ by equation \eqref{Abrundung}. Thus Lemma \ref{Abschaetzung p kein Teiler} yields
$$\mu_f(pq) \geq pq+1-n.$$
Since $\mu_f(p) \leq p-1$, it follows with equation \eqref{mu_f(pq)} that $\mu_f(q) \leq q- \frac{q}{p}$ and therefore $$\mu_f(pq) \leq p \mu_f(q)-1 \leq pq-q-1.$$
Combining these two results we get
$$pq+1-n \leq \mu_f(pq) \leq pq-q-1,$$
which is equivalent to $q \leq n-2$.
The third assertion easily follows from (1) and (2).
\end{proof}

In the homogeneous case, Bhatt and Singh \cite{BS} relate the integer $h$ that appears in Theorem \ref{theorem f-schwelle} to the so-called Hasse invariant.
The aim of the next two sections is to answer the following question: What is $h$ in the quasi-homogeneous case? For this, we first consider the case of a curve and then we pass on to the case $n>2$.

\section{The case of a curve}

In this section let $R=K[x,y,z]$, where $K$ is perfect and let $\mathfrak{m}=(x,y,z)$ be the maximal ideal of $R$. Let$$\deg(x)=\alpha_x, \deg(y)=\alpha_y \text{ and } \deg(z)=\alpha_z$$ and let $f \in R$ be a quasi-homogeneous polynomial of degree $w=\alpha_x+ \alpha_y + \alpha_z$ and type $\alpha=(\alpha_x, \alpha_y, \alpha_z)$ with an isolated singularity. By $$C=\Proj(R/fR)$$ we denote the curve given by $f$. Then by \cite[Proposition 3.3 and 3.4]{Reid} the curve $C$ is in fact projective, thus by \cite[Exposé III, 4.4.]{CL} $C$ is ordinary if and only if the map $$F:H^1(C, \mathcal{O}_C) \rightarrow H^1(C, \mathcal{O}_C)$$ induced by Frobenius is bijective. Since $K$ is perfect, $F$ is bijective if and only if $F$ is injective.

In order to prove the main theorem of this section, we need the following two lemmata.

\begin{lemma} \label{LemmaKoeff}
If the coefficient of $(xyz)^{p-1}$ in $f^{p-1}$ is nonzero, then $\fpt(f)=1$. Otherwise $\fpt(f)=1- \frac{1}{p}$.
\end{lemma}

\begin{proof}
First, suppose that the coefficient of $(xyz)^{p-1}$ in $f^{p-1}$ is nonzero. This means that the coefficient of $\frac{1}{xyz}$ in $\frac{f^{p-1}}{(xyz)^p}$ is nonzero. By Lemma \ref{interpretation mu} this is equivalent to $\mu_f(p)>p-1.$ Since $\mu_f(p) \leq p$, we get that $\mu_f(p)=p$. Therefore, $$\mu_f(q)=q$$ by Lemma \ref{Liften mu_f} (1) and it follows $\fpt(f)=1$.

Now, let the coefficient of $(xyz)^{p-1}$ in $f^{p-1}$ be zero, which implies that the coefficient of $\frac{1}{xyz}$ in $\frac{f^{p-1}}{(xyz)^p}$ is zero. Again, by Lemma \ref{interpretation mu} this is equivalent to $\mu_f(p) \leq p-1$, thus $$\frac{\mu_f(p)}{p} \leq 1-\frac{1}{p} <1.$$ By Lemma \ref{Liften mu_f} (2) it follows that the sequence $\left\{\frac{\mu_f(q)}{q}\right\}_q$ is constant. Since $\fpt(f) \in \left\{1,1-\frac{1}{p}\right\}$ by Theorem \ref{theorem f-schwelle} (3), this gives $\fpt(f)=\frac{\mu_f(p)}{p}=1-\frac{1}{p}$.
\end{proof}

If $f \in K\left[x_0, \ldots, x_n\right]$ is quasi-homogeneous of degree $w=\sum\limits_{i=0}^n \alpha_i$ (where $\alpha_i=\deg \left(x_i \right)$), then a similar argument as above shows: If $p \geq w(n-2)+1$, then $\fpt(f)=1$ if and only if the coefficient of $(x_0 \cdots x_n)^{p-1}$ in $f^{p-1}$ is nonzero.

\begin{lemma}\label{LemmaMainThm}
Let $C=\Proj(R/fR)$ be the curve given by the quasi-homogeneous polynomial $f \in K[x,y,z]$ of degree $w$ and type $\alpha$ with an isolated singularity. Then $H^1(C, \mathcal{O}_C) = \Soc \left( H^2_\mathfrak{m}(R/fR) \right)$.
\end{lemma}

\begin{proof}
Using local cohomology \cite[Theorem 13.21]{24hours}, we get
$$H^1(C, \mathcal{O}_C)=H^2_\mathfrak{m}(R/fR)_0$$  
thus it is enough to show that $H^2_\mathfrak{m}(R/fR) _0=\Soc \left( H^2_\mathfrak{m}(R/fR) \right)$.
By Lemma \ref{interpretation mu} we know that $H^2_\mathfrak{m}(R/fR)$ is a submodule of $H^3_\mathfrak{m}(R)(-w)$ and we know that 
$$\Soc\left(H^{3}_{\mathfrak{m}}(R)(-w)\right)=\left\langle \left[\frac{1}{xyz}\right]\right\rangle,$$
which is the degree zero part of $H^{3}_{\mathfrak{m}}(R)(-w)$.
Since $$\mathfrak{m} \cdot H^2_\mathfrak{m}(R/fR)_0 \subset H^2_\mathfrak{m}(R/fR)_{>0}=0,$$ it follows that $H^2_\mathfrak{m}(R/fR)_0 \subseteq \Soc \left( H^2_\mathfrak{m}(R/fR) \right)$. Furthermore, since $H^2_\mathfrak{m}(R/fR) \neq 0$, it follows 
$$\Soc \left( H^2_\mathfrak{m}(R/fR) \right) = H^2_\mathfrak{m}(R/fR) \cap \Soc \left( H^{3}_{\mathfrak{m}}(R)(-w) \right) \neq 0$$
(see Lemma \ref{interpretation mu}).
Therefore, $H^2_\mathfrak{m}(R/fR) _0 = \Soc \left( H^2_\mathfrak{m}(R/fR) \right)$.
\end{proof}

Using these two lemmata, we deduce the main theorem of this section, which generalizes the two-dimensional case of the main theorem of \cite{BS} to the quasi-homogeneous case.

\begin{theorem} \label{MainTheorem}
Let $C=\Proj(R/fR)$ be the curve given by the quasi-homogeneous polynomial $f \in K[x,y,z]$ of degree $w$ and type $\alpha$ with an isolated singularity. Then
$$\fpt(f)= \begin{cases} 1, &\mbox{if  $C$ is ordinary} \\
1-\frac{1}{p}, & \mbox{otherwise.} \end{cases}$$
\end{theorem}

\begin{proof}
The curve $C$ is ordinary if and only if $F:H^1(C, \mathcal{O}_C) \rightarrow H^1(C, \mathcal{O}_C)$ is injective.
Using Lemma \ref{LemmaMainThm}, we have shown that $C$ is ordinary if and only if $$F: \Soc \left( H^2_\mathfrak{m}(R/fR) \right) \rightarrow \Soc \left( H^2_\mathfrak{m}(R/fR) \right)$$ is injective. But this is equivalent to the fact that $$\widetilde{F^1_1}: H^2_\mathfrak{m}(R/fR) \rightarrow H^2_\mathfrak{m}(R/fR)$$ is injective (see Lemma \ref{interpretation mu}). Equivalently, the coefficient of $\frac{1}{xyz}$ in $\frac{f^{p-1}}{(xyz)^p}$ (which is the coefficient of $(xyz)^{p-1}$ in $f^{p-1}$) is nonzero. By Lemma \ref{LemmaKoeff} the result follows.
\end{proof}

We conclude this section with some examples.

\begin{example} \label{examplePeriod}
Up to permutation there are three solutions $(a,b,c) \in \mathbb{N}^3$ of the equation $\frac{1}{a} + \frac{1}{b} + \frac{1}{c}=1$, namely
$$(3,3,3), (2,4,4) \text{ and } (2,3,6).$$
We will consider the corresponding elliptic singularities that occur in the classification of Arnold directly after the ADE-singularities (see for example \cite{ArnVar}):
$$\begin{aligned}
\tilde{E_6}&=P_8 : &x^3+y^3+z^3+ \lambda xyz=0,\\
\tilde{E_7}&=X_9 : &x^2+y^4+z^4+ \lambda xyz=0,\\
\tilde{E_8}&=J_{10} : &x^2+y^3+z^6+ \lambda xyz=0.
\end{aligned}$$
The aim is to compute the $F$-pure threshold of these three polynomials. In order to do this by Lemma \ref{LemmaKoeff} it is enough to compute the coefficient of $(xyz)^{p-1}$ in the $(p-1)$-th power of the respective polynomial.

Let us start with $f_{\lambda}=x^3+y^3+z^3+ \lambda xyz$, which is (quasi)-homogeneous of degree $3$ and type $\alpha=(1,1,1)$. First, we compute $f_{\lambda}^{p-1}$:
\begin{align*}
&\left(x^3+y^3+z^3+ \lambda xyz\right)^{p-1}\\
&= \sum_{n=0}^{p-1} \binom{p-1}{n} \lambda^{p-1-n} (xyz)^{p-1-n} (x^3+y^3+z^3)^n\\
&= \sum_{n=0}^{p-1} \sum_{k=0}^n \sum_{l=0}^{n-k} \binom{p-1}{n} \binom {n}{k} \binom{n-k}{l} \lambda^{p-1-n} x^{3k+p-1-n} y^{3l+p-1-n} z^{3(n-k-l)+p-1-n}.
\end{align*}
Thus, in order to compute the coefficient of $(xyz)^{p-1}$, we need to solve the following three equations:
$$3k=n, 3l=n \text{ and } 3(n-k-l)=n.$$
Using this, the coefficient of $(xyz)^{p-1}$ in $f_{\lambda}^{p-1}$ is
$$\varphi(\lambda)=\sum_{s=0}^{\left\lfloor \frac{p-1}{3}\right\rfloor} \binom{p-1}{3s} \binom{3s}{s} \binom{2s}{s} 
 \lambda^{p-1-3s}= \sum_{s=0}^{\left\lfloor \frac{p-1}{3}\right\rfloor} \frac{(3s)!}{(s!)^3} (-1)^{3s} \lambda^{p-1-3s},$$
since $\binom{p-1}{3s} \equiv (-1)^{3s} \mod p$.
Therefore,
$$\fpt(f_{\lambda})= \begin{cases} 1, &\mbox{if $\varphi(\lambda) \neq 0$} \\
1-\frac{1}{p}, & \mbox{if $\varphi(\lambda) = 0$.} \end{cases}$$

Now, let us consider $f_{\lambda}=x^2+y^4+z^4+ \lambda xyz$, which is quasi-homogeneous of degree $4$ and type $\alpha=(2,1,1)$.  
Similar to the above we compute that the coefficient of $(xyz)^{p-1}$ in $f_{\lambda}^{p-1}$ is given by 
$$\varphi(\lambda)= \sum_{s=0}^{\left\lfloor \frac{p-1}{4}\right\rfloor} \frac{(4s)!}{(2s)!(s!)^2} \lambda^{p-1-4s}.$$
Therefore,
$$\fpt(f_{\lambda})= \begin{cases} 1, &\mbox{if $\varphi(\lambda) \neq 0$} \\
1-\frac{1}{p}, & \mbox{if $\varphi(\lambda) = 0$.} \end{cases}$$

Lastly, we consider $f_{\lambda}=x^2+y^3+z^6+ \lambda xyz$, which is quasi-homogeneous of degree $6$ and type $\alpha=(3,2,1)$.  
The coefficient of $(xyz)^{p-1}$ in $f_{\lambda}^{p-1}$ is given by 
$$\varphi(\lambda)= \sum_{s=0}^{\left\lfloor \frac{p-1}{6}\right\rfloor} \frac{(6s)!}{(3s)!(2s)!(s!)} \lambda^{p-1-6s}.$$
Therefore,
$$\fpt(f_{\lambda})= \begin{cases} 1, &\mbox{if $\varphi(\lambda) \neq 0$} \\
1-\frac{1}{p}, & \mbox{if $\varphi(\lambda) = 0$.} \end{cases}$$
\end{example}

\begin{remark}
Consider the period
$$\psi\left( \lambda \right):=\frac{1}{\left(2\pi i\right)^3} \oint \frac{\lambda xyz}{f_{\lambda}} \frac{dx}{x} \frac{dy}{y} \frac{dz}{z}$$
of $f_{\lambda}=x^a+y^b+z^c+ \lambda xyz$, where $(a,b,c)$ is one of the three triples of Example \ref{examplePeriod} and $\lambda \neq 0$. One can compute that 
$$\psi\left( \lambda \right)= \sum_{n=0}^{\infty} \left(\frac{-1}{\lambda}\right)^n \left[\left(\frac{x^a+y^b+z^c}{xyz}\right)^n\right]_0,$$
where $\left[-\right]_0$ denotes the degree zero part.
One can show that the corresponding polynomial $$\psi\left( \lambda \right)=\sum_{n=0}^{\lfloor \frac{p-1}{w}\rfloor} \left(\frac{-1}{\lambda}\right)^n \left[\left(\frac{x^a+y^b+z^c}{xyz}\right)^n\right]_0$$ is equal to the polynomial $\varphi\left( \lambda \right)$ computed in Example \ref{examplePeriod}.
\end{remark}

\begin{example}
Next, we want to consider the $T_{a,b,c}$ - singularities given by $$f_{\lambda}=x^a+y^b+z^c+\lambda xyz \text{ with } \frac{1}{a} + \frac{1}{b} + \frac{1}{c}<1 \text{ and } \lambda \neq 0.$$ Since $f_{\lambda}$ is not quasi-homogeneous, we can not use Lemma \ref{LemmaKoeff}. Instead, we remember that the $F$-pure threshold of $f_{\lambda}$ was defined by $\fpt(f_{\lambda})=\lim_{e \rightarrow \infty} \frac{\mu_{f_{\lambda}}(p^e)}{p^e}$ with $\mu_{f_{\lambda}}(p^e)=\min \left\{k \in \mathbb{N} \big| f_{\lambda}^k \in \mathfrak{m}^{[p^e]}\right\}$. In the following we will show that the coefficient of $(xyz)^{q-1}$ in $f_{\lambda}^{q-1}$ is $1$, where $q=p^e$. This means that $f_{\lambda}^{q-1} \notin \mathfrak{m}^{[q]}$, but obviously $f_{\lambda}^q \in \mathfrak{m}^{[q]}$. Thus, $\mu_{f_{\lambda}}(p^e)=p^e$ and therefore $\fpt(f_{\lambda})=1$ for all $\lambda$.

Now, it remains to compute the coefficient of $(xyz)^{q-1}$ in $f_{\lambda}^{q-1}$:
\begin{align*}
&\left(x^a+y^b+z^c+ \lambda xyz\right)^{q-1}\\
&= \sum_{n=0}^{q-1} \binom{q-1}{n} \lambda^{q-1-n} (xyz)^{q-1-n} (x^a+y^b+z^c)^n\\
&= \sum_{n=0}^{q-1} \sum_{k=0}^n \sum_{l=0}^{n-k} \binom{q-1}{n} \binom {n}{k} \binom{n-k}{l} \lambda^{q-1-n} x^{ak+q-1-n} y^{bl+q-1-n} z^{c(n-k-l)+q-1-n}.
\end{align*}
We have to solve the equations
$$ak=n, bl=n \text{ and } c(n-k-l)=n$$
but since $\frac{1}{a} + \frac{1}{b} + \frac{1}{c}<1$, the third equation is never satisfied except for $n=k=l=0$. Therefore, the coefficient of $(xyz)^{q-1}$ in $f_{\lambda}^{q-1}$ is $\varphi(\lambda)=\lambda^{q-1} \equiv 1$. 
\end{example}

\section{The case $n>2$}

Now, let us come back to the situation of Theorem \ref{theorem f-schwelle}. Remember that we consider $R=K[x_0, \ldots, x_n]$ with maximal ideal $\mathfrak{m}=(x_0, \ldots, x_n)$ and let $f \in R$ be a quasi-homogeneous polynomial of degree $w=\sum\limits_{i=0}^n \alpha_i$ with an isolated singularity, where $\alpha_i=\deg \left(x_i\right)$.

Similar to the homogeneous case (\cite{BS}) we want to relate the integer $h$ that appears in Theorem \ref{theorem f-schwelle} to the order of vanishing of the Hasse invariant on some deformation space of $X=\Proj\left(R/fR\right)$. For this, let us first fix some more notation. 

We consider the family $\pi: \mathfrak{X} \rightarrow \Hyp_{w}$ of hypersurfaces of degree $w$ in the weighted projective space $\mathbb{P}^n \left( \alpha_0, \ldots, \alpha_n \right)$.
Our chosen hypersurface $X=\Proj\left(R/fR\right)$ gives a point $\left[X\right]$ in $\Hyp_{w}$.
Set 
$$G := \sum\limits_{i=1}^m \tilde{s}_i g_i \in K[x_0, \ldots, x_n, \tilde{s}_1, \ldots, \tilde{s}_m],$$ where $\left\{g_1, \ldots, g_m\right\} \subset K[x_0, \ldots, x_n]$ is the set of monomials of degree $w$ and we set $\deg(\tilde{s}_i):=0$ for $1 \leq i \leq m$ (such that $G$ is quasi-homogeneous of degree $w$). 
The family $\pi$ of hypersurfaces of degree $w$ in $\mathbb{P}^n \left( \alpha_0, \ldots, \alpha_n \right)$ is given by
\begin{equation} \label{Fakt pi}
\begin{tikzcd}
      \mathfrak{X}=\Proj_{\mathbb{P}^{m-1}}\left( \mathcal{O}_{\mathbb{P}^{m-1}} \left[x_0, \ldots, x_n\right]/ G \right) \arrow{dr}{\pi} \arrow{r}{i} & \Proj_{\mathbb{P}^{m-1}}\left( \mathcal{O}_{\mathbb{P}^{m-1}} \left[x_0, \ldots, x_n\right] \right) 
			\arrow{d} \\
			 & \mathbb{P}^{m-1},
\end{tikzcd}
\end{equation}
where $i$ is a closed immersion.
If $f=\sum\limits_{i=1}^m f_i g_i$, $f_i \in K$, is the defining equation of $X$ in the weighted projective space $\mathbb{P}^n \left( \alpha_0, \ldots, \alpha_n \right)$, then $X$ is the fiber over $\left( f_1, \ldots, f_m \right)$, i.e. $\left[X\right] \in \Hyp_w$ is equal to $V\left(\tilde{s}_i - f_i | i \in \left\{1, \ldots, m\right\}\right)$.
The aim of this section is to prove the following theorem:

\begin{theorem}\label{Hassethm}
If $p \geq w(n-2)+1$, then the integer $h$ in Theorem \ref{theorem f-schwelle} is the order of vanishing of the Hasse invariant at $\left[X\right] \in \Hyp_{w}$ on the deformation space $\mathfrak{X}$ of $X \subset \mathbb{P}^n \left(\alpha_0, \ldots, \alpha_n \right)$ described above.
\end{theorem}

First, let us recall the definition of the \textbf{Hasse invariant} of a family of varieties in characteristic $p$ (see for example \cite{BS}).
Fix a proper flat morphism $\pi: \mathfrak{X} \rightarrow S$ of relative dimension $N$ between noetherian $\mathbb{F}_p$-schemes and assume that the formation of $R^N \pi_{\ast} \mathcal{O}_{\mathfrak{X}}$ is compatible with base change, i.e.
if we have the following pullback diagram
\begin{equation*}
\begin{tikzcd}
      \mathfrak{X} \times_S T \arrow{r}{i^{(1)}} \arrow{d}{\pi^{(1)}} & \mathfrak{X} \arrow{d}{\pi} \\
			 T \arrow{r}{i} & S,
\end{tikzcd}
\end{equation*}
then 
$i^{\ast} R^N \pi_{\ast} \mathcal{O}_{\mathfrak{X}} \cong R^N \pi^{(1)}_{\ast} {i^{(1)}}^{\ast} \mathcal{O}_{\mathfrak{X}}$. Furthermore, assume that 
$R^N \pi_{\ast} \mathcal{O}_{\mathfrak{X}}$ is a line bundle. 

Consider the Frobenius twist $\mathfrak{X}^{(1)}=\mathfrak{X} \times_{\Frob_S} S$ of $\mathfrak{X}$, which gives the following diagram 
\begin{equation*}
\begin{tikzcd}
 \mathfrak{X} \arrow[bend right]{rdd}{\pi} \arrow[dashed]{rd}{\Frob_{\mathfrak{X} / S}} \arrow[bend left]{rrd}{\Frob_{\mathfrak{X}}} & & \\
			 & \mathfrak{X}^{(1)} \arrow{r} \arrow{d}{\pi^{(1)}} & \mathfrak{X} \arrow{d}{\pi} \\
       & S \arrow{r}{\Frob_S} & S.
\end{tikzcd}
\end{equation*} 
By the base change assumption we have
$$R^N \pi_{\ast}^{(1)} \mathcal{O}_{\mathfrak{X}^{(1)}} \cong \Frob_S^{\ast} R^N \pi_{\ast} \mathcal{O}_{\mathfrak{X}} \cong \left( R^N \pi_{\ast} \mathcal{O}_{\mathfrak{X}} \right)^{p}.$$

The relative Frobenius $\Frob_{\mathfrak{X} / S}$ induces a map $\mathcal{O}_{\mathfrak{X}^{(1)}} \rightarrow \left( \Frob_{\mathfrak{X} / S} \right)_{\ast} \mathcal{O}_{\mathfrak{X}}$ and this induces a map
$$H: R^N \pi_{\ast}^{(1)} \mathcal{O}_{\mathfrak{X}^{(1)}} \rightarrow R^N \pi_{\ast}^{(1)} \left( \Frob_{\mathfrak{X} / S} \right)_{\ast} \mathcal{O}_{\mathfrak{X}} \in \Hom \left(R^N \pi_{\ast}^{(1)} \mathcal{O}_{\mathfrak{X}^{(1)}}, R^N \pi_{\ast} \mathcal{O}_{\mathfrak{X}} \right),$$
which is called the Hasse invariant of the family $\pi$ (here we used that $\pi=\pi^{(1)} \circ \Frob_{\mathfrak{X} / S}$).
Since 
\begin{align*}
\Hom \left(R^N \pi_{\ast}^{(1)} \mathcal{O}_{\mathfrak{X}^{(1)}}, R^N \pi_{\ast} \mathcal{O}_{\mathfrak{X}} \right)
 &\cong \Hom \left(\left( R^N \pi_{\ast} \mathcal{O}_{\mathfrak{X}} \right)^{p}, R^N \pi_{\ast} \mathcal{O}_{\mathfrak{X}} \right) \\
&\cong \Hom\left(\mathcal{O}_S, \left( R^N \pi_{\ast} \mathcal{O}_{\mathfrak{X}} \right)^{1-p} \right)\\
&\cong H^0 \left(S, \left( R^N \pi_{\ast} \mathcal{O}_{\mathfrak{X}} \right)^{1-p} \right),
\end{align*}
the Hasse invariant $H$ is a section of a line bundle.

Next, we want to give the definition of the order of vanishing of the Hasse invariant. For this, fix $s \in S$ and an integer $t \geq 0$. Let $t[s]$ be the order $t$ neighbourhood of $s$, i.e. it is defined by the $t$-th power of the ideal defining $s$, and let $t\mathfrak{X}_s \subset \mathfrak{X}$ respectively $t\mathfrak{X}_s^{(1)} \subset \mathfrak{X}^{(1)}$ be the corresponding neighbourhoods of the fibers of $\pi$ respectively of $\pi^{(1)}$, i.e. we have the following cartesian diagrams
\begin{equation*}
\begin{tikzcd}
t\mathfrak{X}_s \arrow{r}{\widetilde{i}} \arrow{d}{\widetilde{\pi}} & \mathfrak{X} \arrow{d}{\pi} \\
       t[s] \arrow[hook]{r}{i} & S
\end{tikzcd}
\hspace{15pt} \text{ and } \hspace{15pt}
\begin{tikzcd}
t\mathfrak{X}^{(1)}_s \arrow{r}{\widetilde{i^{(1)}}} \arrow{d}{\widetilde{\pi^{(1)}}} & \mathfrak{X}^{(1)} \arrow{d}{\pi^{(1)}} \\
       t[s] \arrow[hook]{r}{i} & S.
\end{tikzcd}
\end{equation*} 

The map $\Frob_{\mathfrak{X} / S}$ induces maps $t\mathfrak{X}_s \rightarrow t\mathfrak{X}_s^{(1)}$ for all $t$ and hence maps
$$\phi_t: H^N \left( t\mathfrak{X}_s^{(1)}, \mathcal{O}_{t\mathfrak{X}_s^{(1)}} \right) \rightarrow H^N \left( t\mathfrak{X}_s, \mathcal{O}_{t\mathfrak{X}_s} \right).$$
The two diagrams above demonstrate that the maps $\phi_t$ are given by $i^{\ast}H$, where $i: t[s] \hookrightarrow S$. For this, consider
$$i^{\ast}H : i^{\ast} R^N \pi_{\ast}^{(1)} \mathcal{O}_{\mathfrak{X}^{(1)}} \rightarrow i^{\ast} R^N \pi_{\ast}^{(1)} \left( \Frob_{\mathfrak{X} / S} \right)_{\ast} \mathcal{O}_{\mathfrak{X}}.$$
Here, by the base change assumption, we have
$$i^{\ast} R^N \pi_{\ast}^{(1)} \mathcal{O}_{\mathfrak{X}^{(1)}} = R^N \widetilde{\pi^{(1)}}_{\ast} \left( \widetilde{i^{(1)}} \right)^{\ast} \mathcal{O}_{\mathfrak{X}^{(1)}} = R^N \widetilde{\pi^{(1)}}_{\ast} \mathcal{O}_{t\mathfrak{X}_s^{(1)}} = H^N \left( t\mathfrak{X}_s^{(1)}, \mathcal{O}_{t\mathfrak{X}_s^{(1)}} \right)$$
and similarly
$$i^{\ast} R^N \pi_{\ast}^{(1)} \left( \Frob_{\mathfrak{X} / S} \right)_{\ast} \mathcal{O}_{\mathfrak{X}} = i^{\ast} R^N \pi_{\ast} \mathcal{O}_{\mathfrak{X}} = R^N \widetilde{\pi}_{\ast} \widetilde{i}^{\ast} \mathcal{O}_{\mathfrak{X}} = R^N \widetilde{\pi}_{\ast} \mathcal{O}_{t\mathfrak{X}_s} = H^N \left( t\mathfrak{X}_s, \mathcal{O}_{t\mathfrak{X}_s} \right),$$
since $\pi = \pi^{(1)} \circ \Frob_{\mathfrak{X} / S}$.
Using this identification of $\phi_t$ with $i^{\ast}H$, one obtains the following lemma:

\begin{lemma}
The order of vanishing of the Hasse invariant $H$ at the point $s \in S$ is $\ord_s(H)=\max \left\{t | \phi_t=0\right\}$.
\end{lemma}

Later, we will need the following reformulation of the Hasse invariant (for a proof see \cite{BS}):

\begin{lemma} \label{reformulationHasse}
If $\psi_t: H^N \left( \mathfrak{X}_s, \mathcal{O}_{\mathfrak{X}_s} \right) \rightarrow H^N \left( t\mathfrak{X}_s, \mathcal{O}_{t\mathfrak{X}_s} \right)$ induced by $\Frob_{\mathfrak{X}}$ is nonzero for some $t \leq p$, then $\ord_s(H)+1=\min \left\{ t | \psi_t \neq 0\right\}$. 
\end{lemma}

Now, let us come back to our family $\pi$  of hypersurfaces of degree $w$ in \linebreak $\mathbb{P}^n \left( \alpha_0, \ldots, \alpha_n \right)$. Using diagram (\ref{Fakt pi}), one can check that $\pi$ is a proper morphism of relative dimension $n-1$ between noetherian $\mathbb{F}_p$-schemes 
and Lemma 9.3.4 of \cite{FGAexplained} shows that $\pi$ is also flat. Furthermore, one can prove that $R^{n-1} \pi_{\ast} \mathcal{O}_{\mathfrak{X}}$ 
is a line bundle (see for example \cite[p. 35]{Ogus}). The compatibility of $R^{n-1} \pi_{\ast} \mathcal{O}_{\mathfrak{X}}$ with base change follows from 
\cite[p. 51]{Mumford} or EGA III (\cite[7.7]{EGAIII}), since $H^{n} \left( \mathfrak{X}_s, \mathcal{O}_{\mathfrak{X}_s} \right)=0$.
Thus, the order of vanishing of the Hasse invariant at $\left[X\right] \in \Hyp_w$ on the space of hypersurfaces $\mathfrak{X}$ is defined.

In order to start with the proof of Theorem \ref{Hassethm}, let us first remark that it suffices to consider the affine situation, i.e. we work on the left side of the following diagram
\begin{equation*}
\begin{tikzcd}   
			\Proj_{\mathbb{A}^{m} \setminus 0}\left( \mathcal{O}_{\mathbb{A}^{m} \setminus 0} \left[x_0, \ldots, x_n\right]/ G \right) \arrow{d} \arrow{r} & \Proj_{\mathbb{P}^{m-1}}\left( \mathcal{O}_{\mathbb{P}^{m-1}} \left[x_0, \ldots, x_n\right]/ G \right) \arrow{d} 
			\\
			\left( \mathbb{A}^{m} \setminus 0\right) \times \mathbb{P}^{n}\left( \alpha_0, \ldots, \alpha_n \right) \arrow{d} \arrow{r} & \Proj_{\mathbb{P}^{m-1}}\left( \mathcal{O}_{\mathbb{P}^{m-1}} \left[x_0, \ldots, x_n\right] \right) \arrow{d} \\
			 \mathbb{A}^{m} \setminus 0 \arrow{r} & \mathbb{P}^{m-1}.
\end{tikzcd}
\end{equation*}

Remember that $f=\sum\limits_{i=1}^m f_i g_i$ is the defining equation of $X$ and $G=\sum\limits_{i=1}^m \tilde{s_i} g_i$. 
Changing coordinates via $s_i=\tilde{s}_i-f_i$, one obtains $X':=\pi^{-1} \left(\left[X\right]\right)=\Proj\left(\tilde{R}/(F,s)\tilde{R}\right)$, where $\tilde{R}=K \left[x_0, \ldots, x_n,s_1, \ldots, s_m\right]$, $F=f+\sum\limits_{i=0}^m s_i g_i$ and $s=\left(s_1, \ldots, s_m \right)$. 
Furthermore, let $tX$ respectively $tX'$ be the order $t$ neighbourhoods of $X$ in $\mathbb{P}^n \left(\alpha_0, \ldots, \alpha_n \right)$ respectively $X'$ in $\mathfrak{X}$, i.e. $$tX=\Proj\left(R/f^tR\right) \text{ and } tX'=\Proj_{K[s] / s^t}\left(\tilde{R}/(F,s^t)\tilde{R}\right).$$

\begin{proof}[Proof of Theorem \ref{Hassethm}]
For $1 \leq t \leq p$ we consider the following commutative diagram
\begin{equation*}
\begin{tikzcd}
      \tilde{R}/\left(F,s\right) \arrow{r}{\Frob_{\tilde{R}}} \arrow[-,double equal sign distance]{d} & \tilde{R}/\left(F^p,s^{[p]}\right) \arrow{r}{a} & \tilde{R}/\left(F,s^p\right) \arrow{r}{\tilde{pr}_1} & \tilde{R}/\left(F,s^{t}\right) \arrow{r}{\tilde{pr}_2} & \tilde{R}/\left(F,s\right) \arrow[-, double equal sign distance]{d} \\
R/f \arrow{r}{\Frob_R} & R/f^{[p]} \arrow[-, double equal sign distance]{r} \arrow{u}{h_1} & R/f^{p} \arrow{r}{pr_1} \arrow{u}{h_2} & R/f^t \arrow{r}{pr_2} \arrow{u}{h_3} & R/f,
\end{tikzcd}
\end{equation*}
where the maps $\tilde{pr}_1, \tilde{pr}_2, pr_1$ and $pr_2$ are the evident projections and the map $a$ is given by the inclusion $\left(F^p, s^{[p]}\right) \subset \left(F,s^p\right)$, since $\left(s_1^p, \ldots, s_m^p\right) \subset \left(s_1, \ldots, s_m\right)^p$. The maps $h_1, h_2$ and $h_3$ are defined as follows:
Obviously, we have a map $\varphi_1: R \hookrightarrow \tilde{R} \twoheadrightarrow \tilde{R}/F^p \twoheadrightarrow \tilde{R}/\left(F^p,s^{[p]}\right)$, which induces the map $h_1$ if and only if $f^p \in \Ker\left(\varphi_1\right)=\left(F^p,s^{[p]}\right)$, which follows from
$$f^p=\left(F - \sum_{i=1}^m s_ig_i\right)^p=F^p + \left(-\sum_{i=1}^m s_ig_i\right)^p \in \left(F^p,s^{[p]}\right).$$
Similarly, we have a map $\varphi_2: R \hookrightarrow \tilde{R} \twoheadrightarrow \tilde{R}/F \twoheadrightarrow \tilde{R}/\left(F,s^t\right)$, which induces the map $h_3$ if and only if $f^t \in \Ker\left(\varphi_2\right)=\left(F,s^t\right)$. But this is true, since
$$f^t=\left(F - \sum_{i=1}^m s_ig_i\right)^t=h \cdot F + g \cdot \left(\sum_{i=1}^m s_ig_i\right)^t \in \left(F,s^t\right)$$
for some $h, g \in \tilde{R}$. The same argument for $t=p$ gives $h_2$. 

Passing to cohomology and taking the degree zero parts yields the following commutative diagram
\begin{equation}\label{c_t}
\begin{tikzcd}
     H^n_{\mathfrak{m}}\left(\tilde{R}/\left(F,s\right)\right)_0 \arrow{r}{b_t} \arrow[-,double equal sign distance]{d} & H^n_{\mathfrak{m}}\left(\tilde{R}/\left(F,s^t\right)\right)_0 \\
H^n_{\mathfrak{m}}\left(R/f\right)_0 \arrow{r}{a_t} & H^n_{\mathfrak{m}}\left(R/f^t\right)_0 \arrow{u}{c_t}.
\end{tikzcd}
\end{equation}

In order to prove Theorem \ref{Hassethm}, we now show the following four equalities
\begin{align} \label{Gleichheiten}
\ord_s(H)+1 &=
 \min \left\{t  \;|\; H^{n-1}\left(X', \mathcal{O}_{X'}\right) \stackrel{b_t}{\longrightarrow} H^{n-1}\left(tX', \mathcal{O}_{tX'}\right) \text{ injective }\right\} \notag \\
&= \min \left\{t \;|\; H^{n-1}\left(X, \mathcal{O}_{X}\right) \stackrel{a_t}{\longrightarrow} H^{n-1}\left(tX, \mathcal{O}_{tX}\right) \text{ injective }\right\}\\
&= \min \left\{t \;|\; H^{n}_{\mathfrak{m}}\left(R/fR\right) \stackrel{\widetilde{F^1_t}}{\longrightarrow} H^{n}_{\mathfrak{m}}\left(R/f^tR\right) \text{ injective }\right\} \notag \\
&=h+1 \notag
\end{align}
and then clearly we get $h=\ord_s(H)$ (here we used \cite[Theorem 13.21]{24hours} to replace local cohomology by sheaf cohomology).

The first equality follows by Lemma \ref{reformulationHasse}. For the proof of the third equation we need the following lemma about the injectivity of Frobenius on the negatively graded part of local cohomology modules:

\begin{lemma} \label{Frob injective}
Let $f \in R$ be a quasi-homogeneous polynomial of degree $d$ and type $\alpha$ with an isolated singularity. Let $t \leq p$. Then for $p \geq nd-w-td+1$ the Frobenius action
$$\widetilde{F^1_t}: \left[H^n_{\mathfrak{m}}(R/fR)\right]_{<0} \rightarrow \left[H^n_{\mathfrak{m}}(R/f^tR)\right]_{<0}$$
is injective.
\end{lemma}

\begin{proof}
As in the proof of Lemma \ref{interpretation mu} for $e=1$ we have the following commutative diagram
\begin{equation*}
\begin{tikzcd}
      0 \arrow{r} & \left[H^n_{\mathfrak{m}}(R/fR)\right]_{\leq -1} \arrow{r} \arrow{d}{\widetilde{F^1_t}}  & \left[H^{n+1}_{\mathfrak{m}}(R)\right]_{\leq -d-1} \arrow{r} \arrow{d}{f^{p-t}F} & \ldots \\
      0 \arrow{r} & \left[H^n_{\mathfrak{m}}(R/f^tR)\right]_{\leq -p} \arrow{r} & \left[H^{n+1}_{\mathfrak{m}}(R)\right]_{\leq -dt-p} \arrow{r} & \ldots.
\end{tikzcd}
\end{equation*}
and again it suffices to prove the injectivity of $f^{p-t}F$. 

For this, let $\left[ \frac{g}{(x_0 \cdots x_n)^{q/p}} \right]$ be an element of $\left[H^{n+1}_{\mathfrak{m}}(R)\right]_{\leq -d-1}$, where $g$ is quasi-homogeneous of degree $\deg(g) \leq w \frac{q}{p}-d-1$ for some $q$.
Suppose that $\left[ \frac{g}{(x_0 \cdots x_n)^{q/p}} \right] \in \Ker \left( f^{p-t}F \right)$, which means that $$0=f^{p-t}F \left( \left[ \frac{g}{(x_0 \cdots x_n)^{q/p}} \right] \right)=\left[ \frac{f^{p-t} g}{(x_0 \cdots x_n)^{q}} \right].$$
Therefore, $f^{p-t}g^p \in \mathfrak{m}^{[q]}$.
Let $$k:=\min\left\{ l|f^lg^p \in \mathfrak{m}^{[q]}\right\},$$ then $0 \leq k \leq p-t$. We want to show that $k=0$, so suppose $k \neq 0$. Arguing as in the proof of Lemma \ref{Abschaetzung p kein Teiler}, we get $f^{k-1}g^p J(f) \subset \mathfrak{m}^{[q]}$.
By Lemma \ref{max Grad Jacobi} and Lemma \ref{lemma colon ideal} it follows that $$f^{k-1}g^p \in \left(\mathfrak{m}^{[q]} :_R J(f) \right) \setminus \mathfrak{m}^{[q]} \subset \left(\mathfrak{m}^{[q]} :_R R_{\geq (n+1)d-2w+1} \right) \setminus \mathfrak{m}^{[q]} \subset R_{\geq (q+1)w-(n+1)d},$$
which implies that $(k-1)d + p \deg(g)=\deg(f^{k-1}g^p) \geq (q+1)w-(n+1)d$.
Since $k \leq p-t$ and $\deg(g) \leq w \frac{q}{p}-d-1$, a short computation yields $p \leq nd-w-td$, which is a contradiction. Therefore, $k=0$, which means that $g^p \in \mathfrak{m}^{[q]}$. 
Thus, $\left[ \frac{g}{(x_0 \cdots x_n)^{q/p}} \right]=0$ and therefore, $f^{p-t}F$ is injective.
\end{proof}

Using this, we can now prove the third equality of (\ref{Gleichheiten}).
Since $p \geq w(n-2)+1 \geq w(n-1-t)+1$ for $1 \leq t \leq p$ and using Lemma \ref{Frob injective}, we know that $\widetilde{F^1_t}$ is injective in negative degrees. Then, since $a_t=\left(\widetilde{F^1_t}\right)_0$, the asserted equality follows.

Next, we prove the fourth equation.
By the first part of Theorem \ref{theorem f-schwelle} we know that $\mu_f(p) = p-h$ and this gives the two inequalities $\mu_f(p)>p-h-1=p-(h+1)$ and $\mu_f(p) \leq p-h$. Thus, by Lemma \ref{interpretation mu} it follows that $\widetilde{F^1_{h+1}} $ is injective, but $\widetilde{F^1_h} $ is not injective. Therefore, $h+1= \min \left\{t \big| \widetilde{F^1_t} \text{ injective } \right\}$.

Hence, to prove Theorem \ref{Hassethm} it only remains to show the second equality of (\ref{Gleichheiten}), which is the most difficult part of the proof.
First, to simplify the notation we will write in the following $f^t$ instead of $f^t \cdot R$, $f^{t-1}/f^t$ instead of $f^{t-1} \cdot R/f^t \cdot R$ and so forth.

Now, diagram (\ref{c_t}) shows that it is sufficient to show that $c_t$ is injective for all $t$. For this, consider the following commutative diagram with exact rows
\begin{equation*}
\begin{tikzcd}
    0 \arrow{r} & f^{t-1}/f^t \arrow{r} \arrow{d}{h_4} & R/f^t \arrow{r} \arrow{d}{h_1} & R/f^{t-1} \arrow{r} \arrow{d}{h_1} & 0 \\
0 \arrow{r} & \left(F,s^{t-1}\right)/\left(F,s^t\right) \arrow{r} & \tilde{R}/\left(F,s^t\right) \arrow{r} & \tilde{R}/\left(F,s^{t-1}\right) \arrow{r} & 0.
\end{tikzcd}
\end{equation*}
Since $f^{t-1} \in \left(F, s^{t-1}\right)$, the map $\varphi_3: f^{t-1} \hookrightarrow \left(F, s^{t-1}\right) \twoheadrightarrow \left(F, s^{t-1}\right)/\left(F, s^t\right)$ induces $h_4$.
Passing to local cohomology we obtain the following diagram 
{\scriptsize%
\begin{equation*}
\begin{tikzcd}
       H^{n-1}_{\mathfrak{m}}\left(R/f^{t-1}\right)_0 \arrow{r} \arrow{d} & H^n_{\mathfrak{m}}\left(f^{t-1}/f^t\right)_0 \arrow{r}{a} \arrow{d}{\phi_t} & H^n_{\mathfrak{m}}\left(R/f^t\right)_0 \arrow{r}{d} \arrow{d}{c_t} & H^n_{\mathfrak{m}}\left(R/f^{t-1}\right)_0  \arrow{d}{c_{t-1}}\\
 H^{n-1}_{\mathfrak{m}}\left(\tilde{R}/\left(F,s^{t-1}\right)\right)_0 \arrow{r} & H^n_{\mathfrak{m}}\left(\left(F,s^{t-1}\right)/\left(F,s^{t}\right)\right)_0 \arrow{r}{b} & H^n_{\mathfrak{m}}\left(\tilde{R}/\left(F,s^{t}\right)\right)_0 \arrow{r}{e} & H^n_{\mathfrak{m}}\left(\tilde{R}/\left(F,s^{t-1}\right)\right)_0.
\end{tikzcd}
\end{equation*}
}

The rest of the proof mainly consists of the following three steps: We first show that $b$ is injective and then conclude that it suffices to prove the injectivity of $\phi_t$ for all $t$ in order to show that $c_t$ is injective for all $t$. As a last step, we prove the injectivity of $\phi_t$ for all $t$.

\textbf{Step 1:} The map $b$ is injective: To prove this, we show that $H^{n-1}_{\mathfrak{m}}\left(\tilde{R}/\left(F,s^{t-1}\right)\right)=0$. For $t=1$ this is clear. For $t=2$ we get $$H^{n-1}_{\mathfrak{m}}\left(\tilde{R}/\left(F,s^{t-1}\right)\right)=H^{n-1}_{\mathfrak{m}}\left(\tilde{R}/\left(F,s\right)\right)=H^{n-1}_{\mathfrak{m}}\left(R/f\right)=0.$$ Now let $t \geq 3$ and consider the exact sequence
\begin{equation*}
\begin{xy}
  \xymatrix{
      0 \ar[r] & \left(F,s^{t-1}\right)/\left(F,s^{t}\right) \ar[r] & \tilde{R}/\left(F,s^t\right) \ar[r] & \tilde{R}/\left(F,s^{t-1}\right) \ar[r] & 0.
  }
\end{xy}
\end{equation*}
This gives the long exact sequence
\begin{equation*}
\begin{tikzcd}
    H^{n-1}_{\mathfrak{m}}\left(\left(F,s^{t-1}\right)/\left(F,s^{t}\right)\right) \arrow{r} & H^{n-1}_{\mathfrak{m}}\left(\tilde{R}/\left(F,s^t\right)\right) \arrow{r} & H^{n-1}_{\mathfrak{m}}\left(\tilde{R}/\left(F,s^{t-1}\right)\right).
\end{tikzcd}
\end{equation*}
By induction, we know that $H^{n-1}_{\mathfrak{m}}\left(\tilde{R}/\left(F,s^{t-1}\right)\right)=0$ and in the following we will show that $H^{n-1}_{\mathfrak{m}}\left(\left(F,s^{t-1}\right)/\left(F,s^{t}\right)\right)=0$. Using this and the exact sequence above, it follows that $H^{n-1}_{\mathfrak{m}}\left(\tilde{R}/\left(F,s^{t}\right)\right)=0$. Therefore, it remains to show that \linebreak $H^{n-1}_{\mathfrak{m}}\left(\left(F,s^{t-1}\right)/\left(F,s^{t}\right)\right)=0$. For this, we compute 
\begin{align}\label{isom}
\left(F,s^{t-1}\right)/\left(F,s^{t}\right) \notag
&= \left(F+s^t+s^{t-1}\right)/\left(F+s^{t}\right) \notag \\ 
& \cong \left(s^{t-1}\right)/\left(s^{t-1} \cap \left(F+s^{t}\right)\right) \notag \\
& \cong \left(s^{t-1}/s^{t}\right) / F (s^{t-1}/s^{t})\\ 
& \cong \left(s^{t-1}/s^{t}\right) / f \left(s^{t-1}/s^{t}\right) \notag \\ 
& \cong R/f \otimes_K \left(s^{t-1}/s^{t}\right). \notag
\end{align}
Hence, using \cite[Proposition 7.15]{24hours}, we get
\begin{align*}
H^{n-1}_{\mathfrak{m}}\left(\left(F,s^{t-1}\right)/\left(F,s^{t}\right)\right) 
& \cong H^{n-1}_{\mathfrak{m}}\left(R/f \otimes_K \left(s^{t-1}/s^{t}\right)\right)\\
& \cong H^{n-1}_{\mathfrak{m}}\left(R/f\right) \otimes_K \left(s^{t-1}/s^{t}\right)\\ 
& = 0.
\end{align*}

\textbf{Step 2:} The second step of the proof is to show that it suffices to prove the injectivity of $\phi_t$ for all $t$, in order to prove the injectivity of $c_t$ for all $t$.
Again, we do this by induction on $t$. It is easy to see that $c_1$ is injective if and only if $\phi_1$ is injective. Now, suppose $c_1, \ldots, c_{t-1}$ are injective.  
Then by a diagram chase and using that $b$ and $\phi_t$ are injective, one can prove that $c_t$ is also injective.

\textbf{Step 3:} Now, as a last step, we prove the injectivity of 
$$\phi_t: H^n_{\mathfrak{m}}\left(f^{t-1}/f^t\right)_0 \rightarrow H^n_{\mathfrak{m}}\left(\left(F,s^{t-1}\right)/\left(F,s^t\right)\right)_0.$$ 
For this, consider the projective resolution 
\begin{equation*}
\begin{tikzcd}
      0 \arrow{r} & f^t \arrow{r} & f^{t-1} \arrow{r} & f^{t-1}/f^t \arrow{r} & 0
\end{tikzcd}
\end{equation*}
of $f^{t-1}/f^t$ (remark that $(f^t)$ can be identified with $R(-tw)$).
Tensoring the sequence
\begin{equation*}
\begin{tikzcd}
     0 \arrow{r} & R(-w) \arrow{r}{f} & R \arrow{r} & R/f \arrow{r} & 0
\end{tikzcd}
\end{equation*}
with $s^{t-1}/s^{t}$ over $K$ yields the projective resolution
\begin{equation*}
{\small
\begin{tikzcd}
      0 \arrow{r} & R(-w) \otimes_K \left(s^{t-1}/s^{t}\right) \arrow{r}{f} & R \otimes_K \left(s^{t-1}/s^{t}\right) \arrow{r} & R/f \otimes_K \left(s^{t-1}/s^{t}\right) \arrow{r} & 0
\end{tikzcd}}
\end{equation*}
of $R/f \otimes_K \left(s^{t-1}/s^{t}\right) \cong \left(F,s^{t-1}\right) / \left(F,s^t\right)$ (see (\ref{isom})).
Altogether, we have the following situation
\begin{equation*}
{\small%
\begin{tikzcd}
      0 \arrow{r} & (f^{t-1}) \cdot f \arrow{r} \arrow[dashed]{d}{\theta} & f^{t-1} \arrow{r} \arrow[dashed]{d}{\theta} & f^{t-1}/f^t \arrow{r} \arrow[dashed]{d}{\theta} & 0 \\
0 \arrow{r} &  R(-w) \otimes_K \left(s^{t-1}/s^{t}\right) \arrow{r}{f} & R \otimes_K \left(s^{t-1}/s^{t}\right) \arrow{r} & R/f \otimes_K \left(s^{t-1}/s^{t}\right) \arrow{r} & 0,
\end{tikzcd}}
\end{equation*}
where the map $\theta$ is induced by $h_4$ and sends $f^{t-1}$ to $\left(\sum\limits_{i=1}^m s_ig_i\right)^{t-1}$ considered as an element of $R/f \otimes_K \left(s^{t-1}/s^{t}\right) \cong \left(F,s^{t-1}\right) / \left(F,s^t\right)$,
i.e. if we define $\left\{f_j\right\}$ to be a basis of the polynomials of degree $t-1$ in the variables $y_1, \ldots y_m$ then
$$\theta \left(f^{t-1}\right)= \sum_j c_j f_j \left( g_1, \ldots, g_m\right) \otimes f_j \left(s_1, \ldots, s_m\right)$$
for some coefficients $c_j \in K$.
If we apply the functor $\Hom_R \left(-, R(-w) \right)_0$ to the above diagram, we get a map 
$$\psi_t: \Hom_R\left(R(-w) \otimes_K \left(s^{t-1}/s^{t}\right), R(-w)\right)_0 \rightarrow \Hom \left(f^{t}, R(-w)\right)_0,$$
 which sends a map $\varphi$ to $\varphi \circ \theta$ and induces a map
$$\psi_t: \Ext^1_R \left(\left(F,s^{t-1}\right)/\left(F,s^t\right), R(-w)\right)_0 \rightarrow \Ext^1_R \left( f^{t-1}/f^t, R(-w)\right)_0.$$
Here,
\begin{align*}
\Ext^1_R \left(f^{t-1}/f^t, R(-w)\right)
&=\Hom_R\left(f^t, R(-w)\right) \Big/ f \cdot \Hom_R\left(f^{t-1}, R(-w)\right)
\end{align*}
and
\begin{align*}
\Ext^1_R &\left(\left(F,s^{t-1}\right) / \left(F,s^t\right), R(-w)\right)\\
=&\Hom_R\left(R(-w) \otimes_K \left(s^{t-1}/s^{t}\right), R(-w)\right) \Big/ f \cdot \Hom_R\left(R \otimes_K \left(s^{t-1}/s^{t}\right), R(-w)\right). 
\end{align*}
By the functoriality of the local duality (see \cite[Theorem 3.6.19]{BrunsHerzog}) the map $\psi_t$ is equal to the map $\phi_t^{\vee}$, since $\phi_t$ is induced by the map $\theta$ by passing to local cohomology and then taking the degree zero parts.

Now, the idea is to prove the surjectivity of $\phi_t^{\vee}$ instead of proving the injectivity of $\phi_t$ by using the above description as Ext-modules.
Therefore, we want to examine $\psi_t$ more closely:
\begin{equation*}
\begin{tikzcd}
	\Hom_R\left(R(-w) \otimes_K \left(s^{t-1}/s^{t}\right), R(-w)\right)_0 \arrow{d}{\psi_t} \arrow{r}{\cong} & \left(R(-w) \otimes_K \left(s^{t-1}/s^{t}\right)^{\vee}\right)_0 \\
\Hom \left(f^{t}, R(-w)\right)_0 \arrow{d}{\cong} & \\
 \Hom \left(R(-tw), R(-w)\right)_0 \arrow{d}{\cong} & \\
R((t-1)w)_0 = R_{(t-1)w}, &
\end{tikzcd}
\end{equation*}
where the first vertical isomorphism is given by the identification of $f^t$ with $R(-tw)$ and the second vertical isomorphism is given by $\varphi \mapsto \varphi(1)$.

Using the above, we get
$$\theta \left(rf^t\right)= \theta \left(r \cdot f \cdot f^{t-1}\right)=rf\theta \left(f^{t-1}\right)= rf \sum_j c_j f_j \left( g_1, \ldots, g_m\right) \otimes f_j \left(s_1, \ldots, s_m\right)$$
for some coefficients $c_j \in K$.
For an element $1 \otimes \delta_i \in \left(R(-w) \otimes_K \left(s^{t-1}/s^{t}\right)^{\vee}\right)_0$, where $\left\{\delta_i\right\}$ is a dual basis of $s^{t-1}/s^{t}$, we have
\begin{align*}
\psi_t \left(1 \otimes \delta_i\right)
&= \left[rf^t \mapsto \left(1 \otimes \delta_i\right) \left(\theta \left(rf^t\right)\right) \right] \\
&= \left[rf^t \mapsto \left(1 \otimes \delta_i\right) \left( rf \sum_j c_j f_j \left( g_1, \ldots, g_m\right) \otimes f_j \left(s_1, \ldots, s_m\right) \right) \right]\\
&= \left[rf^t \mapsto rf c_i f_i \left( g_1, \ldots, g_m\right)  \right].
\end{align*}
As an element of $\Hom \left(R(-tw), R(-w)\right)_0$ this map is given by $r \mapsto r c_i f_i \left( g_1, \ldots, g_m\right)$ and via the last vertical isomorphism this map is sent to $c_i f_i \left( g_1, \ldots, g_m\right) \in R_{(t-1)w}$.
With these observations the surjectivity of $\psi_t$ becomes clear, since the set $\left\{f_j\right\}$ forms a basis of the space of polynomials of degree $t-1$.
\end{proof}


\newcommand{\etalchar}[1]{$^{#1}$}

\end{document}